\documentclass{amsart}

\usepackage{amsmath,amsthm,amssymb,enumitem,verbatim,stmaryrd,xcolor,microtype,graphicx}
\usepackage[top=3.4cm,bottom=3.4cm,left=3.2cm,right=3.2cm]{geometry}
\usepackage{tikz-cd}
\usepackage[colorlinks=true]{hyperref}

\theoremstyle{plain}
\newtheorem{thm}{Theorem}[section]
\newtheorem{cor}[thm]{Corollary}
\newtheorem{lem}[thm]{Lemma}
\newtheorem{prop}[thm]{Proposition}

\newtheorem*{thm*}{Theorem}
\newtheorem*{lem*}{Lemma}

\theoremstyle{definition}
\newtheorem{df}[thm]{Definition}
\newtheorem{rem}[thm]{Remark}

\newtheorem{ex}[thm]{Example}
\newtheorem*{df*}{Definition}
\newtheorem*{ex*}{Example}
\newtheorem*{rem*}{Remark}

\theoremstyle{remark}

\newcommand\C{{\mathbb C}}

\newcommand\E{{\mathbb E}}

\renewcommand\H{{\mathbb H}}

\newcommand\R{{\mathbb R}}

\newcommand\Z{{\mathbb Z}}

\newcommand\cM{{\mathcal M}}

\newcommand\cW{{\mathcal W}}

\DeclareMathOperator{\row}{row}

\DeclareMathOperator{\Jac}{Jac}
\DeclareMathOperator{\rank}{rank}

\title{The Jacobian of a Sixth-Root-of-Unity Matroid}

\author{Matthew Baker}
\email{mbaker@math.gatech.edu}
\address{School of Mathematics, Georgia Institute of Technology \\ Atlanta, Georgia 30332-0160, USA}

\author{Changxin Ding}
\email{cding66@gatech.edu}
\address{School of Mathematics, Georgia Institute of Technology \\ Atlanta, Georgia 30332-0160, USA}

\author{Xu Zhuang}
\email{xzhuang8@uci.edu}
\address{Department of Mathematics, University of California-Irvine\\ Irvine, CA 92697, USA}

\begin{document}
\maketitle
\begin{abstract}
The Jacobian group (also called the sandpile group, Picard group, or critical group) of a graph or, more generally, of a regular matroid has been well studied. Sixth-root-of-unity matroids, also called complex unimodular matroids, are generalizations of regular matroids. This paper provides a definition, and establishes some basic properties, of the Jacobian group of a sixth-root-of-unity matroid. 
\end{abstract}

\section{Introduction}
The Jacobian group (also called the sandpile group, Picard group, or critical group) of a graph is a discrete analogue of the Jacobian of a Riemann surface. More generally, one can define the Jacobian group ${\rm Jac}_{\Z}(\cM)$ for any regular matroid $\cM$ in such a way that the Jacobian of a graph $G$ is equal to the Jacobian of the corresponding graphic matroid; see \cite[Section 4.6]{Merino} and \cite[Section 2.1]{BBY}. 

\medskip

Sixth-root-of-unity matroids (henceforth abbeviated as $\sqrt[6]{1}$-matroids) are a natural generalization of regular matroids.
More precisely, let $\mu_6$ be the group of sixth roots of unity in ${\mathbb C}$ and let 
$\H=\mu_6 \cup\{0\}$. 
A $\sqrt[6]{1}$-matroid is a matroid $\cM$ which is representable over ${\mathbb C}$ by a matrix $M$ with the property that every subdeterminant of $M$ lies in $\H$. (This is non-obviously equivalent to requiring that $\cM$ is representable over $\mathbb{C}$ by a matrix for which every $r\times r$ subdeterminant is either zero or else of modulus $1$, so $\sqrt[6]{1}$-matroids are also called complex unimodular matroids; see Remark~\ref{rem:SRUequivalence}.)

\medskip

Given such a matrix $M$, which we refer to as a {\em representation} of $\cM$, we define a corresponding $\E$-module ${\rm Jac}(M)$, where $\E$ is the ring of Eisenstein integers. We call $\Jac(M)$ the {\em Jacobian} of the representation $M$. We show that the order of $\Jac(M)$, considered as an abelian group, is equal to the square of the number of bases of $\cM$, and that $\Jac(M)$ only depends on the {\em equivalence class} of the representation $M$ (cf.~Definition~\ref{equivalent}). 
When $\cM$ is a regular matroid (and hence also a $\sqrt[6]{1}$-matroid), any two representations of $\cM$ are equivalent and thus the Jacobians of any two representations of $\cM$ are isomorphic. It therefore makes sense, in this case, to speak of the Jacobian of $\cM$, and we show that (as abelian groups) $\Jac(\cM) = \Jac_{\mathbb Z}(\cM) \oplus \Jac_{\mathbb Z}(\cM)$.

\medskip

Having made the definition, it is natural to wonder to what extent the choice of a representation matters for $\sqrt[6]{1}$-matroids which are not regular. We show that if a $\sqrt[6]{1}$-matroid $\cM$ is 3-connected, then the Jacobians of any two representations of $\cM$ are isomorphic as $\E$-modules (and therefore {\em a posteriori} as $\Z$-modules), so it makes sense to speak of the Jacobian of $\cM$. However, we present an example 
showing that the Jacobian of a $\sqrt[6]{1}$-matroid is not well-defined in general, even as a $\Z$-module. 


\medskip

Our paper is organized as follows. In Section~\ref{basic}, we define $\sqrt[6]{1}$-matroids and discuss the important concept of unique representability in the context of such matroids. In Section~\ref{main section}, we define the Jacobian, prove the main result, and present the above-mentioned counterexample. In Section~\ref{minor section}, we study the order of the Jacobian, discuss the connection between the Jacobian of a $\sqrt[6]{1}$-matroid and the Jacobian of a regular matroid, and use a certain orthogonal projection operator $P$ to give an equivalent definition of the Jacobian. We also establish a formula expressing $P$ as an average over all bases of certain matrices related to fundamental cocircuits.
In Section~\ref{example}, we present several examples and computations. 

\section{Preliminaries}\label{basic}


Throughout the paper, $\cM$ is a matroid with ground set $E$, its rank is denoted by $r$, and $n=|E|$. Let $\omega=e^{2\pi i/6}$ be a primitive sixth root of unity, and let $\E=\Z[\omega]$ be the ring of \emph{Eisenstein integers}. All matrices in this paper will be over the complex numbers. If $M$ is a matrix, we denote by $\overline{M}$ its complex conjugate and by $M^{\text H}$ its conjugate transpose. 

\begin{df} Let $\H=\{z\in \mathbb{C}:z^6=1\}\cup\{0\}$. 
\begin{enumerate}
    \item A matrix $M$ is called an \emph{$\H$-matrix} if every subdeterminant of $M$ lies in $\H$. In particular, every entry of an $\H$-matrix is in $\H$.
    \item If a matroid $\cM$ is represented over $\mathbb{C}$ by an $\H$-matrix $M$, the matrix $M$ is called an \emph{$\H$-representation} of $\cM$, and the matroid $\cM$ is called a \emph{$\sqrt[6]{1}$-matroid}.
\end{enumerate}
\end{df}
\begin{rem} \label{rem:SRUequivalence}
A matroid $\cM$ is a $\sqrt[6]{1}$-matroid if and only if it is representable over $\mathbb{C}$ by a matrix for which every $r\times r$ subdeterminant is either zero or else of modulus $1$; see \cite[Theorem 8.13]{CHOE2004}. 
\end{rem}

\begin{rem}
From a modern point of view, one can consider $\H$ as a \emph{partial field} or, more generally, as a \emph{pasture}; see \cite{Baker-Oliver2020,Semple1996}. For the purposes of this paper, however, it suffices to consider $\H$ as a multiplicatively closed subset of $\C$.
\end{rem}

The unique representability of regular matroids is the key fact used to show that a regular matroid has a well-defined Jacobian group. Unique representability also plays an important role for $\sqrt[6]{1}$-matroids, so we recall the definition now.

\begin{df}\label{equivalent}
Two $\H$-representations of $\cM$ are said to be \emph{equivalent} if one can be obtained from the other by a sequence of the following operations: multiplying a row or a column by a sixth root of unity; interchanging two rows; interchanging two columns (together with their labels); or pivoting (cf.~\cite[p. 9]{Semple1996}) on a nonzero element. 
\end{df}

\begin{df}\label{unique representability}
Let $\cM$ be a $\sqrt[6]{1}$-matroid. If for any two $r\times n$ $\H$-representations $M_1$ and $M_2$, $M_1$ is equivalent to either $M_2$ or $\overline{M_2}$, then we say $\cM$ is \emph{uniquely representable over $\H$}. 
\end{df}

Regular matroids are a special case of $\sqrt[6]{1}$-matroids. It is well-known that they are uniquely representable over any partial field (or even, more generally, over any pasture). The special case of this fact needed in this paper is:
\begin{thm}[\cite{Baker-Oliver2020,Semple1996}]\label{regular}
Regular matroids are uniquely representable over $\H$. 
\end{thm}

For $\sqrt[6]{1}$-matroids that are not regular, unique representability does not hold in general. However, one still has unique representability under the additional assumption that the matroid is $3$-connected: 

\begin{thm}[\cite{Semple1996}, Theorem 6.4]\label{key lemma}
If a $\sqrt[6]{1}$-matroid $\cM$ is $3$-connected, then  $\cM$ is uniquely representable over $\H$. 
\end{thm}

The notions of $3$-connectedness and of a $2$-sum of matroids are fundamental in matroid theory; for definitions and details, see \cite{Oxley2011b}. Here we recall the following two theorems.

\begin{thm}[\cite{Oxley2011b},Corollary 8.3.4]\label{thm:decom of matroids}
Every matroid that is not 3-connected can be constructed from $3$-connected proper minors of itself by a sequence of the operations of direct sum and 2-sum.
\end{thm}

\begin{thm}[\cite{whittle1997, Semple1996}]\label{2-sum of partial fields}
If $\cM_1$ and $\cM_2$ are $\sqrt[6]{1}$-matroids, then $\cM_1\oplus \cM_2$ and $\cM_1\oplus_2 \cM_2$ are also $\sqrt[6]{1}$-matroids. 
\end{thm}

\begin{rem}\label{2-sum of representation}
The proof of Theorem~\ref{2-sum of partial fields} is constructive: given $\H$-representations of $\cM_1$ and $\cM_2$, respectively, one can explicitly construct $\H$-representations of $\cM_1\oplus \cM_2$ and $\cM_1\oplus_2 \cM_2$; see \cite{Oxley2011b,Semple1996}. 
\end{rem}


\section{The Jacobian of a sixth-root-of-unity matroid}\label{main section}

In this section, we will define the Jacobian of a representation $M$ of a $\sqrt[6]{1}$-matroid $\cM$. We will see that when $\cM$ is uniquely representable over $\H$, its Jacobian is well-defined independent of the choice of a representation. We will also present an example showing that the Jacobian of a $\sqrt[6]{1}$-matroid $\cM$ is not well-defined in general. 

Let $M$ be an $\H$-matrix. We define the $\E$-modules   
\[\Lambda^{*}_{\E}(M)=\row_{\C}(M)\cap \E^{n}\]
and\[\Lambda_{\E}(M)=\row_{\C}(M)^\perp\cap \E^{n},\]
where $\row_{\C}(M)$ is the row space of $M$ and $\row_{\C}(M)^\perp$ is the orthogonal complement of $\row_{\C}(M)$ with respect to the Hermitian inner product. Evidently, we have \[\row_{\C}(M)^\perp=\{v\in \C^n: v\cdot M^{\text H}=0\}.\] 

\begin{df}\label{maindf}
Let $\cM$ be a $\sqrt[6]{1}$-matroid and let $M$ be an $\H$-representation of $\cM$. \begin{enumerate}
    \item We define the $\E$-module\[\Jac(M)=\frac{\E^{n}}{\Lambda_{\E}(M)\oplus \Lambda^{*}_{\E}(M)}.\]
    \item When $\Jac(M)$ is independent of the choice of the representation $M$, we call it the Jacobian of $\cM$, and denote it by $\Jac(\cM)$. 
\end{enumerate}
\end{df}
\begin{rem}
Our definition of the Jacobian mimics the one for regular matroids; see \cite[Section 2.1]{BBY} or Section~\ref{regular Jacobian} below. There is a subtle difference between the real and complex cases, however. For a real matrix $M$, we have $\row_{\R}(M)^\perp=\ker_\R(M)$, but for a complex matrix, we have $\row_{\C}(M)^\perp=\ker_\C(\overline{M})$. If we use $\ker_\C(M)$ instead of $\row_{\C}(M)^\perp$ in the definition, then the two lattices $\Lambda_{\E}(M)$ and $\Lambda^{*}_{\E}(M)$ will not be orthogonal in general.
\end{rem}
We will show that when $\cM$ is uniquely representable over $\H$, the Jacobian of $\cM$ is well-defined. We start with a lemma. 

\begin{lem}\label{row space}
Let $\row_\E(M)$ be the set of linear combinations of the rows of $M$ with coefficients in $\E$. Then 
    \[\Lambda^{*}_{\E}(M)=\row_\E(M).\]
\end{lem}
\begin{proof}
    Clearly, $\Lambda^{*}_{\E}(M)\supseteq\row_\E(M)$. 
    
    For the other direction, let $v\in\Lambda^{*}_{\E}(M)$ and write it as a row vector. Let $M'$ be an $r \times n$ full-rank submatrix of $M$. Then $z\cdot M'=v$ has a solution $z\in\C^r$. Let $M''$ be an $r \times r$ full-rank submatrix of $M'$. Then $z\cdot M''=v'$, where $v'$ is a submatrix of $v$. Because $v\in\E^n$ and $M$ is an $\H$-matrix, we have $z=v'\cdot (M'')^{-1}\in\E^r$. Thus $v=z\cdot M'\in\row_\E(M)$. 
\end{proof}

The following trivial lemma allows us to focus on the case that $M$ has full rank. 
\begin{lem}\label{rref}
If $M$ is an $\H$-matrix representing $\cM$, there exists an $r\times n$ $\H$-representation $M'$ (obtained by removing some rows from $M$) such that  $\Lambda_{\E}(M)=\Lambda_{\E}(M')$ and $\Lambda^{*}_{\E}(M)=\Lambda^{*}_{\E}(M')$. In particular, $\Jac(M)=\Jac(M')$. 
\end{lem}

For regular matroids, the Jacobian group is isomorphic to the cokernel of $MM^{\text H}$; see \cite[Proposition 2.1.2]{BBY}. This property also holds for $\sqrt[6]{1}$-matroids, as we will show momentarily. We remark that for a graph $G$, the matrix $AA^{\text H}$ is called the \emph{Laplacian matrix}, where $A$ is the incidence matrix. The matrix $MM^{\text H}$ here plays a similar role as the graph Laplacian. Because for any row vector $v$, the three equations $vMM^{\text H}=0$, $vMM^{\text H}v^{\text H}=0$, and $vM=0$ are equivalent, we have $\rank(MM^{\text H})=\rank(M)$. In particular, when $M_{r\times n}$ has full row rank, $MM^{\text H}$ has full rank. 

\begin{prop}\label{Laplacian}
Let $M_{r\times n}$ be an $\H$-matrix representing $\cM$. Then we have the following isomorphism of $\E$-modules:
\begin{align*}
\frac{\E^{n}}{\Lambda_{\E}(M)\oplus \Lambda^{*}_{\E}(M)} & \to \frac{\E^{r}}{\row_\E(MM^{\text H})}\\
[v] & \mapsto [v\cdot M^{\text H}].
\end{align*}
\end{prop}
\begin{proof}
Because $M$ is a full-rank $\H$-matrix, the $\E$-module homomorphism 
\begin{align*}
f:\E^{n}& \to \E^{r}\\
v & \mapsto v\cdot M^{\text H}
\end{align*}
is surjective. 

It remains to show $f^{-1}(\row_\E(MM^{\text H}))=\Lambda_{\E}(M)\oplus \Lambda^{*}_{\E}(M)$. For any $v\in f^{-1}(\row_\E(MM^{\text H}))$, $f(v)=vM^{\text H}\in\row_\E(MM^{\text H})$. Hence we may write $vM^{\text H}=uMM^{\text H}$ for some $u\in\E^r$. Then $(v-uM)M^{\text H}=0$, which implies $v-uM\in\Lambda_{\E}(M)$. Therefore $v\in\Lambda_{\E}(M)\oplus \Lambda^{*}_{\E}(M)$. The other direction follows immediately from Lemma~\ref{row space}. 

\end{proof}

For a graph, we may use the Smith normal form (SNF) of the Laplacian matrix to compute its Jacobian group; see \cite{Stanley}. In our case, the matrices $M$ and $MM^{\text H}$ are defined over the ring $\E$. 
Because the usual complex norm makes $\E$ a Euclidean domain, and hence a principal ideal domain (PID), the theory of the Smith Normal Form applies equally well in our case. We recall some classical facts about the Smith Normal Form over a PID:

\begin{thm}\label{SNF}
Let $A$ be a nonzero matrix over a PID $R$. 
\begin{enumerate}
    \item There exist invertible matrices $S$ and $T$ over $R$ such that the product $SAT$ is a diagonal matrix whose main diagonal is $(\alpha_1, \cdots, \alpha_r, 0, \cdots, 0)$, where $r$ is the rank of $A$ and $\alpha_i$ divides $\alpha_{i+1}$ in $R$ for $1\leq i<r$. The elements $\alpha_i$ are unique up to multiplication by a unit and are called the \emph{elementary divisors} of $A$. The diagonal matrix $SAT$ is called the \emph{Smith normal form} of the matrix $A$. We denote the process of going from $A$ to its SNF by 
\begin{center}
\begin{tikzcd}
A\arrow{r}[above]{\text{SNF}} & (\alpha_1, \cdots, \alpha_r, 0, \cdots, 0).
\end{tikzcd}     
\end{center}    
    \item Let $d_i(A)$ be the greatest common divisor of all the $i\times i$ subdeterminants of $A$ for $1\leq i\leq r$, and let $d_0(A)=1$. Then \[\alpha_i=\frac{d_i(A)}{d_{i-1}(A)}.\]
\end{enumerate}
\end{thm}

We adopt the notation $\cong_R$ for $R$-module isomorphisms. 

\begin{cor}[\cite{Stanley}]\label{FGM}
Let $A_{n\times n}$ be a matrix over a PID $R$ and let\begin{tikzcd}
A\arrow{r}[above]{\text{SNF}} & (a_1, \cdots, a_n).
\end{tikzcd}  
Then we have\[\frac{R^{n}}{\row_R(A)}\cong_R \bigoplus_{i=1}^nR/(a_i).\]
\end{cor}

Combining Proposition~\ref{Laplacian} and Corollary~\ref{FGM}, we get the following result. 

\begin{cor}\label{Jac-Smith}
Let $M_{r\times n}$ be an $\H$-representation. Let\begin{tikzcd}
MM^{\text H}\arrow{r}[above]{\text{SNF}} & (\alpha_1, \cdots, \alpha_r).
\end{tikzcd} Then we have\[\Jac(M)\cong_\E \bigoplus_{i=1}^r\E/(\alpha_i).\]
\end{cor}

The next two results will allow us to prove that if a $\sqrt[6]{1}$-matroid $\cM$ is uniquely representable, then $\Jac(\cM)$ is well-defined independent of the choice of a representation.

\begin{lem}\label{equivlent Laplacian}
If $M_1$ and $M_2$ are equivalent $\H$-representations, then $M_1M_1^{\text H}$ and $M_2M_2^{\text H}$ have the same SNF. 
\end{lem}
\begin{proof}
It is easy to see that the column operations on a matrix $M$ in Definition~\ref{equivalent} does not change $MM^{\text H}$. Thus we may assume $M_1=SM_2$, where $S$ is an invertible matrix over $\E$. Hence $M_1M_1^{\text H}=SM_2M_2^{\text H}S^{\text H}$. Because $S^{\text H}$ is also invertible over $\E$, $M_1M_1^{\text H}$ and $M_2M_2^{\text H}$ have the same SNF by Theorem~\ref{SNF}.  
\end{proof}

\begin{lem}\label{conjugate Laplacian}
Let $M$ be a matrix over $\E$. \begin{enumerate}
    \item The two matrices $MM^{\text H}$ and $\overline{M}(\overline{M})^{\text H}$ have the same SNF.
    \item Let\begin{tikzcd}
MM^{\text H}\arrow{r}[above]{\text{SNF}} & (a_1, \cdots, a_n).
\end{tikzcd}Then $a_i^6\in\Z$ for any $i$.
\end{enumerate} 
\end{lem}
\begin{proof}
Note that $\overline{M}(\overline{M})^{\text H}$ is equal to the transpose of $MM^{\text H}$. By Theorem~\ref{SNF}, a matrix and its transpose have the same SNF, and hence the first assertion holds.  

Clearly\begin{tikzcd}
\overline{M}(\overline{M})^{\text H}\arrow{r}[above]{\text{SNF}} & (\overline{a_1}, \cdots, \overline{a_n}).
\end{tikzcd}By the uniqueness of the elementary divisors, for any $i$ we have $a_i=u\overline{a_i}$, where $u^6=1$. Thus $a_i^2=u\overline{a_i}a_i=u|a_i|^2$, which implies  $a_i^6\in\Z$.
\end{proof}
\begin{rem}
In general, $M$ and $\overline{M}$ do not have the same SNF. For example, the $1\times 1$ matrices $[2+\omega]$ and $[2+\overline{\omega}]$ have different SNFs. As $\E$-modules, $\E/(2+\omega)\ncong_\E\E/(2+\overline{\omega})$, although $\E/(2+\omega)\cong_\Z\E/(2+\overline{\omega})$. 
\end{rem}

By Lemma~\ref{rref}, Corollary~\ref{Jac-Smith}, Lemma~\ref{equivlent Laplacian}, and Lemma~\ref{conjugate Laplacian}, we have the following theorem.

\begin{thm}\label{Main}
Let $\cM$ be a $\sqrt[6]{1}$-matroid that is uniquely representable over $\H$. \begin{enumerate}
    \item The Jacobian $\Jac(\cM)$ of $\cM$ is well-defined.  
    \item  Let $M_{r\times n}$ be an $\H$-representation of $\cM$, where $r$ is the rank of $\cM$ and $n$ is the number of elements. If we let
    \begin{tikzcd}MM^{\text H}\arrow{r}[above]{\text{SNF}} & (\alpha_1, \cdots, \alpha_r),\end{tikzcd}then\[\Jac(\cM)\cong_\E \bigoplus_{i=1}^r\E/(\alpha_i)\]as $\E$-modules. Moreover, for any $i$, we have $\alpha_i^6\in\Z\backslash\{0\}$.    
\end{enumerate}
\end{thm}


As a consequence of Theorem~\ref{regular} and Theorem~\ref{key lemma}, we have the following corollaries of Theorem~\ref{Main}. 

\begin{cor}\label{regular-Jac}
If $\cM$ is regular matroid, then $\Jac(\cM)$ is well-defined. 
\end{cor}

\begin{cor}
If $\cM$ is a $3$-connected $\sqrt[6]{1}$-matroid, then $\Jac(\cM)$ is well-defined. 
\end{cor}

We now present an example (which was not so easy to find!) showing that the Jacobian of a $\sqrt[6]{1}$-matroid is not well-defined in general. 
\begin{prop}\label{nonexample}
There exists a $\sqrt[6]{1}$-matroid admitting two $\H$-representations $M$ and $M'$ such that $\Jac(M)\ncong_\E\Jac(M')$.    
\end{prop}

\begin{proof}
Consider the two $\H$-matrices 
\[M_1=
\begin{bmatrix}
1&0&0&0&-\omega^2&\omega^2&1&1 \\
0&1&0&0&\omega^2&-\omega^2&0&0 \\
0&0&1&0&1&-1&1&1 \\ 
0&0&0&1&-\omega^2&\omega&0&0 
\end{bmatrix}\]
and
\[M_8=\left[
\begin{array}{*{21}c}
1&0&0&0&-\omega^2&\omega^2 &1&1&1&1&1&1&1&1&1 \\
0&1&0&0&\omega^2&-\omega^2 &0&0&0&0&0&0&0&0&0 \\
0&0&1&0&1&-1               &1&1&1&1&1&1&1&1&1 \\ 
0&0&0&1&-\omega^2&\omega   &0&0&0&0&0&0&0&0&0 
\end{array}
\right].\]

We identify the first column of $M_1$ and the first column of $M_8$ as the basepoint $p$. Consider the $2$-sums $M=M_1\oplus_2 M_8$ and $M'=M_1\oplus_2 \overline{M_8}$ with respect to $p$ (see Remark~\ref{2-sum of representation}).
Then $M$ and $M'$ are two $\H$-representations of the same matroid $\cM$. With the aid of a computer, we have
\begin{center}
\begin{tikzcd}
MM^{\text H}\arrow{r}[above]{\text{SNF}} & (1,1,1,1,1,21,147),
\end{tikzcd}
\begin{tikzcd}
M'(M')^{\text H}\arrow{r}[above]{\text{SNF}} & (1,1,1,1,1,3,1029).
\end{tikzcd}
\end{center}
By Corollary~\ref{Jac-Smith},  $\Jac(M)\cong_\E\E/(21)\oplus\E/(147)$ and $\Jac(M')\cong_\E\E/(3)\oplus\E/(1029)$.

For the computer code, see 

\url{https://cocalc.com/share/public_paths/bc224d34aa04bda4720a433e8599cd268c66b0db}.

\end{proof}

\begin{rem}
In the above example, $\Jac(M)$ and $\Jac(M')$ are not even isomorphic as $\Z$-modules (abelian groups).  
\end{rem}



\section{Other results}\label{minor section}
\subsection{The order of the Jacobian}

Kirchhoff's celebrated Matrix-Tree Theorem gives a determinantal formula for the number of spanning trees of a graph $G$ in terms of the Laplacian matrix $M$ of $G$. This formula implies that the cardinality of $\Jac(G)$ is equal to the number of spanning trees of $G$. The same proof (which involves just the unimodularity of $M$ as an input to the Cauchy-Binet formula) shows that the cardinality of the Jacobian group of a regular matroid $\cM$ is equal to the number of bases of $\cM$. Since representations of $\sqrt[6]{1}$-matroids are also unimodular, it is no surprise that a similar result holds in our setting. The only difference is that, since $\E$ has rank 2 as a module over $\Z$, our formula involves the {\em square} of the number of bases. The following is implicitly proved in \cite[p. 7]{Lyons}, following the outline just discussed. 

\begin{prop}\label{base}
Let $M$ be an $\H$-representation of a $\sqrt[6]{1}$-matroid $\cM$. Then the number of bases of $\cM$ is equal to $\det(MM^{\text H})$. 
\end{prop}

\begin{cor} \label{cor:sizeofJacobian}
Let $M$ be an $\H$-representation of a $\sqrt[6]{1}$-matroid $\cM$. Then the order of $\Jac(M)$ is equal to the square of the number of bases. 
\end{cor}

\begin{proof}
By Corollary~\ref{Jac-Smith}, $|\Jac(M)|=\prod_{i=1}^r|\alpha_i|^2=|\det(MM^{\text H})|^2$. The conclusion follows from Proposition~\ref{base}. 
\end{proof}

\subsection{The connection with Jacobian group of regular matroid}\label{regular Jacobian}
Recall that if a matroid $\cM$ is represented by a totally unimodular matrix $M$, then $\cM$ is called a \emph{regular matroid}. The \emph{Jacobian group} of a regular matroid is defined by
\[\Jac_{\Z}(\cM)=\frac{\Z^n}{\Lambda_{\Z}(M)\oplus\Lambda^{*}_{\Z}(M)}(=:\Jac_\Z(M)),\] 
where $n$ is the number of elements of $\cM$, $\Lambda_{\Z}(M)=\ker_{\R}(M)\cap \Z^{n}$, and $\Lambda^{*}_{\Z}(M)=\row_{\R}(M)\cap \Z^n$. Due to the unique representability of regular matroids, this is well-defined; see \cite[Section 2.1]{BBY} and \cite[Section 3.1]{Wager2010}.  


By Corollary~\ref{regular-Jac}, we also have an $\E$-module $\Jac(\cM)$. Since any $\E$-module is naturally a $\Z$-module, we may also view $\Jac(\cM)$ as an abelian group. 

\begin{prop}
Let $\cM$ be a regular matroid. Then we have the following isomorphism of abelian groups:
\[\Jac(\cM)\cong_\Z \Jac_{\Z}(\cM)\oplus \Jac_{\Z}(\cM).\]
\end{prop}
\begin{proof}

Let $M_{r\times n}$ be a totally unimodular integer matrix representing $\cM$. Then $M$ is also an $\H$-matrix. 

We have the following $\Z$-module isomorphism:
\begin{align*}
    f: \E^n &\to \Z^n\oplus \Z^n \\ 
    (a_1+b_1\omega,\cdots,a_n+b_n\omega) &\mapsto ((a_1,\cdots,a_n),(b_1,\cdots,b_n)).
\end{align*}
It is easy to check that\[f(\Lambda_{\E}(M))=\Lambda_{\Z}(M)\oplus\Lambda_{\Z}(M)\]and \[f(\Lambda^*_{\E}(M))=\Lambda^*_{\Z}(M)\oplus\Lambda^*_{\Z}(M).\]
Hence $f$ induces the desired isomorphism.
\end{proof}

\subsection{An equivalent definition of the Jacobian and the orthogonal projection}

For a regular matroid represented by a totally unimodular integer matrix $M$, we have isomorphisms \[\Jac_\Z(M)\cong_\Z\frac{\Lambda_{\Z}(M)^\#}{\Lambda_{\Z}(M)}\]
and
\[\Jac_\Z(M)\cong_\Z\frac{\Lambda^*_{\Z}(M)^\#}{\Lambda^*_{\Z}(M)},\]
where $\Lambda^\#$ denotes the \emph{dual lattice} of $\Lambda$, that is, \[\Lambda^\#=\{v\in\Lambda\otimes\R:\langle v,u\rangle\in\Z \; \forall \; u\in\Lambda\}.\]
The isomorphisms are naturally induced by the orthogonal projections of $\R^n$ onto $\ker_{\R}(M)$ and $\row_{\R}(M)$, respectively; see \cite[p. 10]{BBY} and \cite[Lemma 1 of Section 4]{BLN}. For a graph $G$, let $M$ be the incidence matrix with respect to a fixed reference orientation. Then the two orthogonal projections can be written as a sum of matrices indexed by the spanning trees of $G$, respectively; see \cite[Proposition 6.3 and 7.3]{Biggs} and \cite[Theorem 1]{Maurer}. We will generalize these results to $\sqrt[6]{1}$-matroids. We will also give a product formula for the orthogonal projections (cf. Proposition~\ref{proj}).



Let $M_{r\times n}$ be an $\H$-matrix representing a $\sqrt[6]{1}$-matroid $\cM$, where $r$ is the rank of $\cM$. 
Recall that we have the orthogonal decomposition \[\C^n=\row_{\C}(M)\oplus\row_{\C}(M)^\perp\]with respect to the Hermitian inner product. The $\E$-modules $\Lambda^{*}_{\E}(M)=\row_{\C}(M)\cap \E^{n}$ and $\Lambda_{\E}(M)=\row_{\C}(M)^\perp\cap \E^{n}$ are viewed as $\E$-lattices, and their dual $\E$-lattices are defined by
\[\Lambda^*_\E(M)^\#=\{v\in\row_\C(M):v\cdot u^{\text H}\in\E \; \forall \; u\in\Lambda^*_{\E}(M)\}\]
and
\[\Lambda_\E(M)^\#=\{v\in\row_{\C}(M)^\perp:v\cdot u^{\text H}\in\E \; \forall \; u\in\Lambda_{\E}(M)\},\]
respectively, where the vectors are written as row vectors as in Section~\ref{main section}. 

We focus on $\Lambda^{*}_{\E}(M)=\row_\E(M)$. By a similar argument, one can get the results for $\Lambda_{\E}(M)$. 

\begin{prop}\label{proj}Let
\[P:\C^n\to\row_{\C}(M)\] be the \emph{orthogonal projection} of $\C^n$ to $\row_{\C}(M)$, i.e., $P$ is the $\C$-linear map such that $P |_{\row_{\C}(M)^\perp}$ is the zero map and $P |_{\row_{\C}(M)}$ is the identity map. Then:
\begin{enumerate}
\item (Product formula for the orthogonal projection operator) $P(v)=vM^{\text H}(MM^{\text H})^{-1}M$ for any $v\in\C^n$.
\item $P(\E^n)=\Lambda^{*}_{\E}(M)^\#$.
\item $P$ induces the $\E$-module isomorphism \[\Jac(M)=\frac{\E^n}{\Lambda_{\E}(M)\oplus\Lambda^{*}_{\E}(M)}\cong_\E\frac{\Lambda^{*}_{\E}(M)^\#}{\Lambda^{*}_{\E}(M)}.\]
\end{enumerate}
\end{prop}

\begin{proof}\leavevmode
\begin{enumerate}
\item The orthogonal projection $P$ is unique. For any $v\in\row_\C(M)^\perp$, $vM^{\text H}(MM^{\text H})^{-1}M=0$ by definition. Because  $MM^{\text H}(MM^{\text H})^{-1}M=M$, we have $vM^{\text H}(MM^{\text H})^{-1}M=v$ for any $v\in\row_\C(M)$.
\item By the first part of the proposition, we can factor the map $P$ as
\[
\begin{tikzcd}[column sep=huge]
P: \C^n\arrow[r,twoheadrightarrow, "\cdot M^{\text H}"] & \C^r\arrow[r, "\cdot (MM^{\text H})^{-1}", "\cong" '] &\C^r\arrow[r, "\cdot M", "\cong" '] & \row_\C(M),
\end{tikzcd}
\]where the first map is a surjection and the other two are isomorphisms. 

Then by restricting the map $P$ to $\E^n$, we have $\E$-module homomorphisms \[
\begin{tikzcd}[column sep=large]
P|_{\E^n}: \E^n\arrow[r,twoheadrightarrow, "\cdot M^{\text H}"] & \E^r\arrow[r, "\cdot (MM^{\text H})^{-1}", "\cong" '] &\{v\in\C^r:vMM^{\text H}\in\E^r\}\arrow[r, "\cdot M", "\cong" '] & \{v\in\row_\C(M):vM^{\text H}\in\E^r\},
\end{tikzcd}
\]where the first homomorphism is a surjection and the other two are isomorphisms. Hence $P(\E^n)$ equals $\{v\in\row_\C(M):vM^{\text H}\in\E^r\}$, which is $\Lambda^{*}_{\E}(M)^\#$.

\item By the factorization of $P|_{\E^n}$ in (2), we get the following isomorphisms:

\[
\begin{tikzcd}[column sep=large]
\dfrac{\E^n}{\Lambda_{\E}(M)\oplus\Lambda^{*}_{\E}(M)}\arrow[r, "\cdot M^{\text H}", "\cong" '] & \dfrac{\E^r}{\row_\E(MM^{\text H})}\arrow[r, "\cdot (MM^{\text H})^{-1}", "\cong" '] &\dfrac{\{v\in\C^r:vMM^{\text H}\in\E^r\}}{\E^r}\arrow[r, "\cdot M", "\cong" '] & \dfrac{\Lambda^{*}_{\E}(M)^\#}{\Lambda^{*}_{\E}(M)},
\end{tikzcd}
\]where the leftmost isomorphism comes from Proposition~\ref{Laplacian} and the other two are trivial.
\end{enumerate}
\end{proof}

\subsection{Averaging formula for the orthogonal projection operator}
In this section we give another formula for the orthogonal projection operator $P$, this time as an average over all bases of $\cM$ of certain simple operators related to the fundamental cocircuits of $\cM$. Our main result is Proposition~\ref{matrixN}, which we view as a complex analogue of \cite[Proposition 6.3]{Biggs}. We adopt the same proof technique as in \cite{Biggs}. 

We first introduce some notation. For a matrix $N$ and a set $J$ indexing certain columns of $N$, we denote by $N[J]$ the submatrix of $N$ formed by the columns indexed by $J$. We often use a basis $B$ of $\cM$ as the index set of the columns of $M$ corresponding to $B$.

\begin{df}\label{NB}
Let $B$ be a basis of $\cM$.
\begin{enumerate}
    \item Denote $M_B=(M[B])^{-1}M$. In particular, $M_B[B]$ is the identity matrix. 
    \item Define $N_B$ to be the $n\times n$ matrix whose restriction to the rows indexed by $B$ is $M_B$, and whose other rows are zero.  
    \item Let $e\in B$. The $e$-th row of $N_B$ is denoted by  $f_B(e)$. 
\end{enumerate}
\end{df}

\begin{rem}
The matrix $M_B$ is known as a \emph{standard representative matrix} for $\cM$ with respect to $B$. The vector $f_B(e)$ should be viewed as the ``signed'' fundamental cocircuit of $e$ with respect to $B$; see Remark~\ref{funda-cocircuit}. More precisely, 
the $e^{\rm th}$ row $f_B(e)$ of $N_B$ is an $\H$-cocircuit (in the sense of \cite{Baker-Bowler2019}) of the $\H$-matroid associated to $M$; it is in fact the unique $\H$-cocircuit with value 1 at $e$ whose support is the fundamental cocircuit of $e$ with respect to $B$.
\end{rem}

We establish some basic properties of $N_B$.

\begin{lem}\label{unique vector}
Let $B$ be a basis of $\cM$. Then for any row vector $u\in\C^r$, there exists a unique $v\in\row_\C(M)$ such that $v[B]=u$. 
\end{lem}
\begin{proof}
We write $v=xM$, where $x\in\C^r$. Because $v[B]=x(M[B])$, $x$ is determined by the linear equation $x(M[B])=u$. Because $B$ is a basis, the matrix $M[B]$ is invertible, and hence the equation has a unique solution $x$. 
\end{proof}

\begin{lem}\label{easy half}
Let $B$ be a basis of $\cM$. \begin{enumerate}
    \item For any row vector $v\in\row_\C(M)^\perp$, $N_Bv^{\text H}=0$. 
    \item For any row vector $v\in\row_\C(M)$, $vN_B=v$.
\end{enumerate}
\end{lem}
\begin{proof}\leavevmode
\begin{enumerate}
    \item Because every non-zero row of $N_B$ is a row of $M_B$, it is in $\row_\C(M)$. By definition, $N_Bv^{\text H}=0$. 
    \item Note that $vN_B\in\row_\C(M)$ because every row of $N_B$ is in $\row_\C(M)$. By Lemma~\ref{unique vector}, we only need to check \[(vN_B)[B]=v[B].\] By Definition~\ref{NB}, we have $vN_B=v[B]M_B$, and hence \[(vN_B)[B]=(v[B]M_B)[B]=(v[B])(M_B[B])=v[B].\] 
\end{enumerate}

\end{proof}

The proof of the following result relies on Lemma~\ref{Hermitian}, which asserts that $N=N^H$. 
As the proof of Lemma~\ref{Hermitian} is somewhat technical, we postpone it to the end of this section. 

\begin{prop}\label{matrixN}
Let \[N=\sum_{B\text{ basis of } \cM}N_B\]
and write $\tilde{N} = \frac{1}{\kappa} N,$
where $\kappa$ is the number of bases of $\cM$.
Then the orthogonal projection operator $P$ (defined in Proposition~\ref{proj}) satisfies $P(v)=v \tilde{N}$ for any $v\in\C^n$.
\end{prop}

\begin{proof}
It is enough to show that 
    \[v(\kappa^{-1}N)=\begin{cases} 
      v & v\in\row_\C(M) \\
      0 & v\in\row_\C(M)^\perp. \\
\end{cases}\]
The first identity follows from Lemma~\ref{easy half}, and the second by combining Lemma~\ref{easy half} with the fact that $N=N^{\text H}$ (which we will show in Lemma~\ref{Hermitian}).
\end{proof}

\begin{rem}
By Proposition~\ref{proj} and Proposition~\ref{matrixN}, we have the identity $N=M^{\text H}(MM^{\text H})^{-1}M$. We leave it as a (hard) exercise to prove the identity directly using linear algebra. This can be used to give another proof of Proposition~\ref{matrixN}.
\end{rem}

The rest of this section is devoted to proving that $N=N^{\text H}$. 
For a row vector $v\in\C^n$ and $e\in E$, we denote by $v[e]$ the $e$-th entry of $v$. 

\begin{lem}\label{fbi1}
The row vector $f_B(e)$ is the unique vector $v\in\row_\C(M)$ such that $v[e]=1$ and $v[B\backslash \{e\}]=0$. Moreover, all the entries of $f_B(e)$ are in $\H$. 
\end{lem}
\begin{proof}
By Definition~\ref{NB}, $f_B(e)$ satisfies the conditions of the first assertion. The uniqueness is guaranteed by Lemma~\ref{unique vector}. Since $M$ is an $\H$-matrix, so is $M_B$. As a row of $M_B$, $f_B(e)$ is therefore defined over $\H$. 
\end{proof}

We abbreviate $(B\backslash\{i\})\cup\{j\}$ as $B-i+j$, where $B\subseteq E$ and $i,j\in E$. 
\begin{lem}\label{fbi2}
Let $B$ be a basis of $\cM$, $i$ an element of $B$, and $j$ an element of $E$. Then $f_B(i)[j]\neq 0$ if and only if $B-i+j$ is a basis of $\cM$.   
\end{lem}
\begin{proof}
When $j\in B$, the statement is clearly true. So we may assume that $j\notin B$. 
Note that $M_B$ is a representation of $\cM$, $M_B[B]$ is the identity matrix, and $f_B(i)$ is a row of $M_B$. We get that $f_B(i)[j]\neq 0$ if and only if $\det(M_B[B-i+j])\neq 0$ by expanding the determinants along the row $f_B(i)$. 
\end{proof}
\begin{rem}\label{funda-cocircuit}
Let $B$ be a basis of $\cM$ and $i$ an element of $B$. Then the set $C^*:=\{j\in E: B-i+j\text{ is a basis of } \cM \}$ is a cocircuit of $\cM$, called the
\emph{fundamental cocircuit} of $i$ with respect to $B$. By Lemma~\ref{fbi2}, the support of $f_B(i)$ is exactly this fundamental cocircuit. 
\end{rem}

\begin{lem}\label{exchange}
Let $i$ and $j$ be two \emph{distinct} elements of $E$.
\begin{enumerate}
    \item The map 
    \begin{align*}
    \{B\text{ basis}: i\in B, f_B(i)[j]\neq 0\}& \to \{B\text{ basis}: j\in B, f_B(j)[i]\neq 0\}\\
    B & \mapsto B-i+j
    \end{align*}
    is a bijection.     
    \item When a basis $B$ satisfies $i\in B$ and $f_B(i)[j]\neq 0$, the complex numbers $f_B(i)[j]$ and $f_{B-i+j}(j)[i]$ are conjugate. 
\end{enumerate}
\end{lem}

\begin{proof}\leavevmode
\begin{enumerate}
    \item By Lemma~\ref{fbi2}, it is equivalent to show that the map 
    \begin{align*}
    \{B\text{ basis}: i\in B, j\notin B, B-i+j\text{ is a basis}\}& \to \{B\text{ basis}: j\in B, i\notin B, B-j+i\text{ is a basis}\}\\
    B & \mapsto B-i+j
    \end{align*}
    is a bijection. This is true because $B\mapsto B-j+i$ is the inverse map (which is a purely set-theoretic fact, it does not require any property of matroids). 
    \item Denote $z=f_{B-i+j}(j)[i](\neq 0)$ and $v=z^{-1}f_{B-i+j}(j)$. Because $v\in\row_\C(M)$, $v[i]=1$, and $v[B\backslash \{e\}]=0$, we have $v=f_B(i)$ by the first half of Lemma~\ref{fbi1}. Hence $f_B(i)[j]=v[j]=z^{-1}=(f_{B-i+j}(j)[i])^{-1}.$ By the second half of Lemma~\ref{fbi1}, the absolute values of the complex numbers $f_B(i)[j]$ and $f_{B-i+j}(j)[i]$ are $1$, so they are conjugate. 

\end{enumerate}
\end{proof}

\begin{lem}\label{Hermitian}
The matrix $N$ is Hermitian.
\end{lem}
\begin{proof}
Let $i$ and $j$ be two distinct elements of $E$. Denote the $(i,j)$ entry of $N$ by $n_{ij}$. We must show that $n_{ij}=\overline{n_{ji}}$. By Definition~\ref{NB}, we have
\[n_{ij}=\sum_{i\in B}f_B(i)[j]\text{ and }n_{ji}=\sum_{j\in B}f_B(j)[i].\] By Lemma~\ref{exchange}, it follows that $n_{ij}=\overline{n_{ji}}$.
\end{proof}



\section{Examples}\label{example}

All the matroids in this section are $3$-connected $\sqrt[6]{1}$-matroids. The relevant computer code can be found via the link in the proof of Proposition~\ref{nonexample}. 

\begin{ex}The Jacobian of the uniform matroid $U_{2,4}$.

It is not hard to check that $U_{2,4}$ is $3$-connected and can be represented by the $\H$-matrix
$M=\begin{bmatrix}
1 & 0 &1 &1 \\
0 &1 &1 & \omega
\end{bmatrix}$. Then $MM^{\text H}=\begin{bmatrix}
3 & 1+\overline{\omega}\\ 
1+\omega & 3
\end{bmatrix}$. 

Making use of Theorem~\ref{SNF}(2), we find\begin{tikzcd}
MM^{\text H}\arrow{r}[above]{\text{SNF}} & (1+\omega, 2(1+\omega)).
\end{tikzcd}     

So, $\Jac(U_{2,4}) \cong_\E \E/(1+\omega) \oplus \E/(2(\omega+1))$. 

As abelian groups, $\Jac(U_{2,4}) \cong_\Z\Z/6\Z\oplus\Z/6\Z$, so the order of $\Jac(U_{2,4})$ is $36$. Since a basis for $U_{2,4}$ is just a 2-element subset of a 4-element set, the number of bases of $U_{2,4}$ is $6$.
\end{ex}

Let $\cM$ be a simple $\sqrt[6]{1}$-matroid of rank $r$. By \cite[Theorem 2.1]{OVW}, when $r\neq 3$, the number of edges of $\cM$ is bounded by $\binom{r+2}{2} -2$ and the bound is attained if and only if $\cM$ is $T_r$; when $r=3$, the number of edges of $\cM$ is bounded by $9$ and the bound is attained if and only if $\cM$ is $AG(2,3)$. We compute the Jacobians of these extremal $\sqrt[6]{1}$-matroids. 

\begin{ex}\label{AG} The Jacobian of the affine geometry $AG(2,3)$.

For the definition of the matroid $AG(2,3)$, see \cite[p. 170]{Oxley2011b}. It is $3$-connected and can be represented by the $\H$-matrix\[
M=\begin{bmatrix}
1 & 0 & 0 & 1 & 0 & 1 & 1 & 1 & 1 \\
0 & 1 & 0 & 1 & 1 & 0 & \overline{\omega} & 1 & \overline{\omega} \\
0 & 0 & 1 & 0 & 1 & -\omega & -\omega & \overline{\omega} & \overline{\omega}
\end{bmatrix};
\]see \cite[p. 653]{Oxley2011b} and \cite[p. 597]{whittle1997}. 
With the aid of a computer, we have
\[
MM^{\text H}=\begin{bmatrix}
6 & 2+2\omega & -2+4\omega\\
2+2\overline{\omega} & 6   & 2+2\omega\\
-2+4\overline{\omega} & 2+2\overline{w} & 6
\end{bmatrix}\begin{tikzcd}
\arrow{r}[above]{\text{SNF}} & (2+2\omega, 2+2\omega,6)
\end{tikzcd}.     
\]

So, $\Jac(AG(2,3))\cong_\E\E/(2+2\omega)\oplus\E/(2+2\omega)\oplus\E/(6)$. 

As abelian groups, 
$\Jac(AG(2,3)) \cong_\Z (\Z/2\Z)^2 \oplus (\Z/6\Z)^4$. 
So the order of $\Jac(AG(2,3))$ is $72^2$. Making use of \cite[Proposition 6.2.3]{Oxley2011b}, one finds that the number of bases of $AG(2,3)$ is indeed $72$, as predicted by Corollary~\ref{cor:sizeofJacobian}.
\end{ex}

\begin{ex} The Jacobian of $AG(2,3)\backslash e$ (Example~\ref{AG} continued).

Every single-element deletion is isomorphic to the same matroid, denoted by $AG(2,3)\backslash e$; see \cite[p. 653]{Oxley2011b}. With the aid of a computer, we know it is $3$-connected, and $\Jac(AG(2,3)\backslash e) \cong_\E \E/(2(1+\omega)) \oplus \E/(8(1+\omega))$.
\end{ex}

\begin{ex} The Jacobian of $T_r$.

Let $r$ be a positive integer. Let $I_r$ denote the identity matrix of size $r$. For $r\geq 2$, let $D_r$ denote the $r\times \binom{r}{2}$ matrix whose columns consist of all $r$-tuples with exactly two non-zero entries, the first equal to $1$ and the second equal to $-1$. Recall that $T_r$ is the matroid represented by the matrix
\[M_r=\left[
\begin{array}{c|c|c|c|c}
1 & 0 \cdots 0  & 1  \cdots 1 & \omega \cdots \omega & 0  \cdots 0\\
\hline 
\begin{array}{c}
     0  \\
     \vdots\\
     0
\end{array} & I_{r-1} & I_{r-1} & I_{r-1} & D_{r-1}
 \end{array}
\right].\]

For a geometric definition of $T_r$, see \cite[p. 166]{OVW}.

If we replace every $\omega$ in $M_r$ with a transcendental number $\alpha$, then by \cite[Lemma 3.1]{OVW} we will obtain a \emph{near-unimodular} matrix, i.e., a matrix whose non-zero subdeterminants are in $\{\pm \alpha^i(1-\alpha)^j:i,j\in\Z\}$. As a direct consequence, $M_r$ is an $\H$-matrix and $T_r$ is a $\sqrt[6]{1}$-matroid. 

The matroid $T_2$ is just $U_4^2$. 
\begin{lem}
$T_r$ is $3$-connected.     
\end{lem}
\begin{proof}

By definition (cf. \cite[p. 293]{Oxley2011b}), we need to show that $T_r$ has neither $1$-separations nor $2$-separations. We will only present the argument for $2$-separations, as the argument for $1$-separations is similar but easier. 

Let $\{X,Y\}$ be a partition of the ground set of $T_r$ such that $|X|\geq 2$ and $|Y|\geq 2$. We need to prove that \[\rank(X)+\rank(Y)\geq r+2.\] 

If $\rank(X)=r$ or $\rank(Y)=r$, then the above inequality clearly holds. The case where $r$ equals $1$ or $2$ is trivial. From now on, we assume $\rank(X)<r$, $\rank(Y)<r$, and $r\geq 3$.

Denote by $J$ the set of elements of $T_r$ which correspond to the columns of $M_r$ not containing $\omega$. Note that the restriction of $T_r$ to $J$ is the graphic matroid $\cM(K_{r+1})$ of the complete graph $K_{r+1}$. Because $\rank(X\cap J)<r$ and $r\geq 3$, we have $|Y\cap J|
\geq 2$. Similarly, we have $|X\cap J|\geq 2$. Because $K_{r+1}$ is $3$-connected, by \cite[Proposition 8.1.9]{Oxley2011b}, $\cM(K_{r+1})$ is also $3$-connected. Hence  $\rank(X\cap J)+\rank(Y\cap J)\geq r+2$, which implies $\rank(X)+\rank(Y)\geq r+2$. 

\end{proof}

We omit a detailed proof of the following result, as the computations involved are rather elaborate.

\begin{prop}\leavevmode
\begin{enumerate}
    \item We have $\det(M_rM_r^{\text H})=3r(r+2)^{r-2}$. Consequently, the number of bases of $T_r$ is $$3r(r+2)^{r-2}.$$
    \item For $r\geq 3$, we have\[
M_rM_r^{\text H}\begin{tikzcd}
\arrow{r}[above]{\text{SNF}} & \:
\end{tikzcd}
\begin{cases} 
      (1,1,r+2,\cdots,r+2,3r(r+2)) & r\equiv 1,3 \pmod{6} \\
      (1,1,r+2,\cdots,r+2,(1+\omega)(r+2),(1+\omega)r(r+2)) & r\equiv 5  \pmod{6} \\
      (1,2,r+2,\cdots,r+2,3r(r+2)/2) & r\equiv 0,4  \pmod{6} \\
      (1,2,r+2,\cdots,r+2,(1+\omega)(r+2),(1+\omega)r(r+2)/2) & r\equiv 2  \pmod{6}.\\
\end{cases}
\]
\end{enumerate}

\end{prop}

For example,
\[M_3M_3^{\text H}\begin{tikzcd}
\arrow{r}[above]{\text{SNF}} & (1,1,45), 
\end{tikzcd}\]
\[M_4M_4^{\text H}\begin{tikzcd}
\arrow{r}[above]{\text{SNF}} & (1,2,6,36), 
\end{tikzcd}\]
\[M_5M_5^{\text H}\begin{tikzcd}
\arrow{r}[above]{\text{SNF}} & (1,1,7,7(1+\omega),35(1+\omega)), 
\end{tikzcd}\]
\[M_6M_6^{\text H}\begin{tikzcd}
\arrow{r}[above]{\text{SNF}} & (1,2,8,8,8,72),
\end{tikzcd}\]
\[M_7M_7^{\text H}\begin{tikzcd}
\arrow{r}[above]{\text{SNF}} & (1,1,9,9,9,9,189), 
\end{tikzcd}\]
\[M_8M_8^{\text H}\begin{tikzcd}
\arrow{r}[above]{\text{SNF}} & (1,2,10,10,10,10,10(1+\omega),40(1+\omega)). 
\end{tikzcd}\]

The idea of the proof is that, by applying elementary row and column operations, we can replace $M_rM_r^{\text H}$ by the following matrix without changing the SNF:

\[L_r=\left[
\begin{array}{c|c|c}
2r+2 & 1+\omega  & 0 \cdots 0  \\
\hline 
r(1+\overline{\omega}) & 0  & 0 \cdots 0  \\
\hline 
\begin{array}{c}
     0  \\
     \vdots\\
     0
\end{array} & 
\begin{array}{c}
     1  \\
     \vdots\\
     1
\end{array} & (r+2)I_{r-2}
 \end{array}
\right].\]

Then we compute all the subdeterminants  of $L_r$ of size $1,2,3,r-1$ and $r$. By Theorem~\ref{SNF}, we obtain the SNF of $L_r$. 

The Jacobian of $T_r$ can be derived directly from the SNF. 
\end{ex}

\begin{ex} The Jacobian of the whirl $\cW^r$.

Another interesting family of matroids is the whirl $\cW^r$, where $r\geq 2$ is the rank. It is a $3$-connected $\sqrt[6]{1}$-matroid; see \cite[p. 659]{Oxley2011b}. 

We sketch the process of computing the Jacobian of $\cW^r$ as follows. 

By \cite[Proposition 8.8.7]{Oxley2011b}, the matrix
\[M_r=\left[
\begin{array}{c|c}
I_r &
\begin{array}{cccccc}
1 & 0 & 0 & \cdots & 0 & \alpha  \\
1 & 1 & 0 & \cdots & 0 & 0  \\
0 & 1 & 1 & \cdots & 0 & 0  \\
\vdots & \vdots & \vdots & \ddots & \vdots & \vdots  \\
0 & 0 & 0 & \cdots & 1 & 0  \\
0 & 0 & 0 & \cdots & 1 & 1  
\end{array}
\end{array}
\right]\]
is representation of $\cW^r$ over all fields if $\alpha\notin\{0,(-1)^r\}$. To obtain an $\H$-representation, we set \[\alpha=\begin{cases} 
      \omega & r\text{ is even} \\
      \omega^2 & r\text{ is odd}. \\
\end{cases}\]Then we show that \[L_r:=M_rM_r^{\text H}=\left[\begin{array}{cccccc}
3 & 1 & 0 & \cdots & 0 & \alpha  \\
1 & 3 & 1 & \cdots & 0 & 0  \\
0 & 1 & 3 & \cdots & 0 & 0  \\
\vdots & \vdots & \vdots & \ddots & \vdots & \vdots  \\
0 & 0 & 0 & \cdots & 3 & 1  \\
\overline{\alpha} & 0 & 0 & \cdots & 1 & 3  
\end{array}\right]\]
and
\[\det(L_r)=f_{2r+2}-f_{2r-2}-1,\]where $f_n$ is the $n$-th Fibonacci number (e.g. $f_6=8$).  

We have the following formula for the SNF (and hence a corresponding formula for Jacobian). 

\begin{prop}
\[
L_r\begin{tikzcd}
\arrow{r}[above]{\text{SNF}} & \:
\end{tikzcd}
\begin{cases} 
      (1,\cdots,1, f_{2r+2}-f_{2r-2}-1) & r\equiv 0,1,3 \pmod{4} \\
      (1,\cdots,1,1+\omega,\dfrac{f_{2r+2}-f_{2r-2}-1}{1+\omega}) & r\equiv 2  \pmod{4}.
\end{cases}
\]
\end{prop}
The key step in the proof of the proposition is that \[
\gcd(L_r(r,1),L_r(r,r))=
\begin{cases} 
      1 &  r\equiv 0,1,3 \pmod{4} \\
      1+\omega & r\equiv 2  \pmod{4},
\end{cases}
\]where the greatest common divisor is computed in $\E$ and $L_r(i,j)$ denotes the subdeterminant of $L_r$ obtained by deleting the $i$-th row and $j$-th column. 
\end{ex}

\begin{ex} The Jacobian of $Q_{10}$.

The matriod $Q_{10}$ is a $10$-element, rank-$5$, self-dual matroid. It is a $\sqrt[6]{1}$-matroid, and is a splitter for the class of $\sqrt[6]{1}$- matroids; see Sage 10.0 Reference Manual \cite{sage}. 

The matroid $Q_{10}$ has $181$ bases; see the computer code. By Proposition~\ref{base}, $\det(MM^{\text H})=181$, where $M$ is any $5\times 10$ $\H$-matrix representing $Q_{10}$. Note that $181$ is a prime natural number and $181\equiv 1  \pmod{3}$, so $181$ can be factored into two conjugate primes in $\E$; see \cite{wikipedia}. By Theorem~\ref{SNF}, \[\begin{tikzcd}MM^{\text H}
\arrow{r}[above]{\text{SNF}} & (1,1,1,1,181),
\end{tikzcd}\]and hence $\Jac(Q_{10}) \cong_\E \E/(181)$.

\end{ex}



\section*{Acknowledgements}
We thank Tianyi Zhang for his assistance with several of the computer-aided examples, and Charles Semple and Zach Walsh for helpful discussions. 
The first author was supported by NSF grant DMS-2154224 and a Simons Fellowship in Mathematics.

\bibliography{SRU}
\bibliographystyle{plain}

\end{document}